\documentclass[11pt]{amsart}

\usepackage[margin=1.12in]{geometry}
\usepackage{amsmath,amssymb,amsthm,mathtools,cite}
\usepackage[hidelinks]{hyperref}
\usepackage{enumitem}
\usepackage{microtype}

\numberwithin{equation}{section}

\newtheorem{theorem}{Theorem}[section]
\newtheorem{proposition}[theorem]{Proposition}
\newtheorem{lemma}[theorem]{Lemma}
\newtheorem{corollary}[theorem]{Corollary}
\theoremstyle{definition}
\newtheorem{definition}[theorem]{Definition}
\newtheorem{hypothesis}[theorem]{Hypothesis}
\theoremstyle{remark}
\newtheorem{remark}[theorem]{Remark}

\newlist{subconditions}{enumerate}{1}
\setlist[subconditions]{
	label=\textup{(\thetheorem.\arabic*)},
	ref=\thetheorem.\arabic*,
	leftmargin=4.2em,
	labelsep=0.65em,
	itemsep=0.5\baselineskip,
	topsep=0.5\baselineskip
}

\DeclareMathOperator{\Var}{Var}
\newcommand{\R}{\mathbb R}
\newcommand{\E}{\mathbb E}
\newcommand{\B}{\mathcal B}
\newcommand{\A}{\mathcal A}
\newcommand{\D}{\mathcal D}
\newcommand{\Renew}{\mathcal R}
\newcommand{\Tail}{\mathcal T}
\newcommand{\C}{\mathcal C}
\newcommand{\1}{\mathbf 1}
\newcommand{\norm}[1]{\left\lVert #1\right\rVert}
\newcommand{\abs}[1]{\left\lvert #1\right\rvert}
\newcommand{\floor}[1]{\left\lfloor #1\right\rfloor}
\newcommand{\Sh}{\mathrm H}
\newcommand{\given}{\,\vert\,}

\allowdisplaybreaks

\title[Survivor renewal for open intermittent maps]{Survivor-conditioned renewal laws and observable bounds for open intermittent maps}
\author{Jason Duvall}
\date{}

\begin{document}

\begin{abstract}
	Recent numerical computations and stochastic modeling by Brevitt and Klages suggest that introducing a hole in a Pomeau--Manneville map can suppress survivor-conditioned Lyapunov stretching. We prove a deterministic renewal theorem which explains this phenomenon and its observable-level generalizations. For an open intermittent map induced on a base away from the neutral fixed point, we describe the asymptotic distribution of the number of completed survivor returns to the base, conditioned on survival up to time $t$. The limiting law is expressed in terms of the killed induced transfer operator; for the conditionally invariant density of the killed induced system it is geometric. We then prove two reward results for additive observables. A reward domination theorem gives bounded survivor-conditioned Birkhoff sums, while a stronger final-tail asymptotic gives convergence to a finite limit. For generalized Pomeau--Manneville maps, bounded observables satisfying $\abs{\psi(x)}\le Cx^\gamma$ near the neutral fixed point and a mild variation condition satisfy the domination hypotheses. When the neutral branch and final tails satisfy the corresponding regularity assumptions, asymptotically regular observables satisfy the convergence hypotheses. In particular, $\psi=\log\abs{f'}$ gives bounded survivor-conditioned Lyapunov stretching for the generalized class; under these additional regularity assumptions, it converges. Under an additional entropy-domination assumption, we also derive a zero entropy-rate consequence for survivor return-length names and record the complementary linear growth of stretching when the hole contains a neighborhood of the neutral fixed point.
\end{abstract}

\maketitle

\section{Introduction}

Intermittent interval maps, originating in the Pomeau--Manneville intermittency mechanism~\cite{PomeauManneville1980,GaspardWang1988}, combine expansion away from a neutral fixed point with long return intervals spent near it, producing polynomial return tails, sublinear stretching, and sometimes infinite invariant measures; see, for example,~\cite{Aaronson1997,LSV1999,Young1999,Gouezel2004}. For a closed system, one natural way to measure stretching or other accumulated quantities is through Birkhoff sums
\begin{equation*}
	S_n\psi(x)=\sum_{k=0}^{n-1}\psi(f^kx).
\end{equation*}
The Lyapunov stretching cocycle corresponds to $\psi=\log\abs{f'}$; for almost every orbit segment on which $f^n$ is differentiable, this sum is $\log\abs{(f^n)'(x)}$. Operator-renewal methods are a standard tool for studying nonuniformly expanding systems and intermittent dynamics; see, for example,~\cite{MelbourneTerhesiu2012,Sarig2002}. The renewal scheme below uses the same broad operator-renewal philosophy, but for survivor-conditioned orbit segments in an open system.

The present paper is motivated by recent work of Brevitt and Klages~\cite{BrevittKlages2025} on open Pomeau--Manneville maps. Their computations and stochastic model suggest that introducing a hole can suppress survivor-conditioned Lyapunov stretching. More precisely, they study finite-time quantities of the form
\begin{equation}\label{eq:intro-bk-profile}
	\E\left[S_n\log\abs{f'}\mid \tau_H>t\right],
	\qquad 0\le n\le t,
\end{equation}
where $\tau_H$ is the hitting time of the hole. We prove a deterministic renewal theorem for the corresponding survivor-conditioned statistics. The theorem is formulated for general observables first; Lyapunov stretching is then obtained as a particular application.

The main result is stated for an induced open system. Let $Y$ be an inducing base away from the neutral fixed point, let $R$ be the first return time to $Y$, and let $F=f^R$ be the induced map, defined almost everywhere on $Y$. The completed survivor-return sets are encoded by a killed induced transfer operator $\A$. The final unfinished excursion at the observation time is encoded separately by a scalar functional $\Tail_t$.

The tail functional satisfies
\begin{equation}\label{eq:intro-tail-functional}
	\Tail_t(v)=a_t\ell(v)+o(a_t),
	\qquad a_t\asymp (1+t)^{-\kappa},
\end{equation}
where $\ell:\B\to\R$ is a positive continuous functional on a suitable Banach space $\B$ of integrable functions on $Y$. Survival up to time $t$ decomposes into finitely many completed survivor returns followed by a final unfinished excursion at time $t$. If $v\in\B$ is a nonnegative density on $Y$ with $\ell((I-\A)^{-1}v)>0$, and if $N_t$ denotes the number of completed survivor returns to $Y$ up to time $t$, then our survivor-renewal theorem gives, for every fixed $r\ge0$,
\begin{equation}\label{eq:intro-main-law}
	\mathbb P_v(N_t=r\mid \tau_H>t)
	\longrightarrow
	\frac{\ell(\A^r v)}{\ell((I-\A)^{-1}v)}.
\end{equation}
The formula gives the limiting distribution of the number of completed survivor returns in this decomposition. If $h$ is a conditionally invariant density for the killed induced operator with eigenvalue $\lambda\in(0,1)$, this limiting law is geometric:
\begin{equation}\label{eq:intro-geometric}
	\mathbb P_h(N_t=r\mid \tau_H>t)
	\longrightarrow
	(1-\lambda)\lambda^r.
\end{equation}
This gives a deterministic counterpart to the geometric killing law appearing in the stochastic model of~\cite{BrevittKlages2025}.

The same renewal structure gives bounds and, under an additional regularity assumption, limits for survivor-conditioned Birkhoff sums. For observables whose completed-excursion rewards and final-tail rewards are dominated at the correct scale, the conditional expectations
\begin{equation}\label{eq:intro-observable-bound}
	\E_v[S_t\psi\mid\tau_H>t]
\end{equation}
remain bounded. If, in addition, the final-tail reward has an asymptotic of the form
\begin{equation}\label{eq:intro-reward-tail}
	\Tail_{\psi,t}(v)=a_t\ell_\psi(v)+o(a_t),
\end{equation}
then the conditional expectations converge to a finite limit. This separation is important: a size estimate for $\psi$ near the neutral fixed point gives boundedness, but convergence requires a limiting asymptotic profile.

For generalized Pomeau--Manneville maps with neutral exponent $\gamma>0$, the domination result applies to bounded observables satisfying
\begin{equation}\label{eq:intro-vanish}
	\abs{\psi(x)}\le Cx^\gamma
	\qquad\text{near }0,
\end{equation}
together with a mild variation condition. Convergence requires additional asymptotic/profile regularity. Suppose first that the observable is asymptotically regular in the sense that
\begin{equation}\label{eq:intro-psi-regular}
	\frac{\psi(x)}{x^\gamma}\to \mathrm{const}
	\qquad(x\to0).
\end{equation}
If, in addition, the neutral branch is asymptotically regular and the final tails satisfy the corresponding profile regularity, then the survivor-conditioned Birkhoff sums converge. The Lyapunov stretching observable satisfies the domination condition because
\begin{equation}\label{eq:intro-log-growth}
	\log\abs{f'(x)}\asymp x^\gamma
	\qquad(x\to0).
\end{equation}
If the neutral branch is asymptotically regular, for instance if $f''(x)/x^{\gamma-1}\to c_f>0$, and the final tails satisfy the corresponding profile regularity, then
\begin{equation}\label{eq:intro-log-asymp}
	\log\abs{f'(x)}\sim \frac{c_f}{\gamma}x^\gamma,
\end{equation}
and the survivor-conditioned Lyapunov stretching converges. Since $\log\abs{f'}\ge0$ for the models considered here, boundedness at the diagonal time also implies the two-parameter boundedness in \eqref{eq:intro-bk-profile}.

We also record one consequence and one complementary case. Under an additional entropy-domination hypothesis, the renewal framework gives sublinear symbolic complexity for survivor return-length names. In the complementary case where the hole contains a neighborhood of the neutral fixed point, survivor orbit segments avoid that neighborhood and therefore remain in a uniformly expanding region; in this case Lyapunov stretching grows at least linearly.

The relation with existing work on open intermittent systems is complementary. Demers and Fernandez proved polynomial escape and singular limiting distributions for intermittent maps with holes away from the neutral fixed point~\cite{DemersFernandez2016}. Demers and Todd developed a thermodynamic description of slow and fast escape regimes for open intermittent maps~\cite{DemersTodd2017}. These works describe escape, limiting survivor distributions, conditionally invariant structures, and pressure. The present result refines that picture at the level of survivor itineraries and accumulated rewards. In particular, convergence of pushed-forward survivor measures to the neutral fixed point does not determine the number of completed returns up to time $t$, nor the Birkhoff sum accumulated along the surviving path.

The paper is organized as follows. Section~\ref{sec:abstract} introduces the abstract induced survivor-renewal scheme and states the standing hypotheses on the killed return operators and tail functionals. Section~\ref{sec:renewal} proves the fixed-$r$ asymptotic and the limiting law for $N_t$ under these hypotheses. Section~\ref{sec:observables} proves reward bounds and convergence for additive observables under the corresponding reward hypotheses. Section~\ref{sec:entropy} derives the zero entropy-rate consequence for return-length names under an entropy-domination hypothesis. Section~\ref{sec:mp} verifies the standing hypotheses and reward bounds for generalized Pomeau--Manneville maps for holes bounded away from the neutral fixed point and satisfying the induced-hole assumptions stated there. It also gives convergence for asymptotically regular behavior near the neutral fixed point, asymptotically regular observables, and final tails with the corresponding profile regularity. Section~\ref{sec:neutral-hole} treats holes containing a neighborhood of the neutral fixed point.

\section{The induced survivor-renewal scheme}\label{sec:abstract}

\subsection{Notation}

We write $b_t\sim c_t$ to mean $b_t/c_t\to1$ as $t\to\infty$, and $b_t\asymp c_t$ to mean that $b_t/c_t$ is bounded above and below by positive constants. Throughout this section, hitting and return times are allowed to take the value $+\infty$, with the convention $\inf\emptyset=\infty$.

\subsubsection{Base and survivor excursions}

We first describe the survivor excursions in the original system. A completed survivor return is an excursion that returns to the base before entering the hole; an unfinished survivor tail is an excursion that has not yet returned by the observation time.

Let $(X,m)$ be a measure space, and let $f:X\to X$ be nonsingular, so $f$ preserves $m$-null sets. Let $Y\subset X$ be an inducing base with $0<m(Y)<\infty$. We write
\begin{equation}\label{eq:normalized-base-measure}
	m_Y\coloneqq\frac{m|_Y}{m(Y)}
\end{equation}
for normalized reference measure on $Y$. Let
\begin{equation}\label{eq:return-time}
	R(y)\coloneqq\inf\{n\ge1:f^n(y)\in Y\}
\end{equation}
be the return time to $Y$. We assume $R<\infty$ for $m_Y$-almost every $y\in Y$, and define, for $m_Y$-almost every $y\in Y$,
\begin{equation}\label{eq:induced-map}
	F(y)\coloneqq f^{R(y)}(y).
\end{equation}

The induced transfer operator $\mathcal L_F$ is defined by
\begin{equation}\label{eq:transfer-duality}
	\int_Y (\mathcal L_F v)w\,dm_Y
	=
	\int_Y v(w\circ F)\,dm_Y
\end{equation}
for suitable observables $w$. In the interval-map case this is equivalently given, on monotone branches, by
\begin{equation}\label{eq:transfer-branch}
	(\mathcal L_F v)(x)
	=
	\sum_{y\in F^{-1}(x)}\frac{v(y)}{\abs{F'(y)}}.
\end{equation}

Let $H\subset X$ be a measurable set, which we call the hole, and define the hitting time
\begin{equation}\label{eq:hitting-time}
	\tau_H(x)\coloneqq\inf\{n\ge0:f^n(x)\in H\}.
\end{equation}
For $n\ge1$, define the return-level set $Y_n$ by
\begin{equation}\label{eq:return-partition}
	Y_n\coloneqq\{y\in Y:R(y)=n\}.
\end{equation}
For the fixed hole $H$, define the survivor-return set
\begin{equation}\label{eq:survivor-return-branch}
	Y_n^{\mathrm{surv}}
	\coloneqq
	Y_n\cap\{\tau_H>n\}.
\end{equation}
Thus $Y_n^{\mathrm{surv}}$ consists of points whose excursion returns to $Y$ at time $n$, while avoiding $H$ up to and including that return time. For $n\ge0$, define the survivor tail set
\begin{equation}\label{eq:unfinished-tail-set}
	Y_{>n}^{\mathrm{surv}}
	\coloneqq
	\{y\in Y:R(y)>n,\ \tau_H(y)>n\}.
\end{equation}
This is the set of points whose $f$-orbit has not returned to $Y$ by time $n$ and has avoided $H$ through time $n$.

\subsubsection{Completed-return operators}

The completed survivor returns are encoded by transfer operators on the base. The index $n$ records the length of one completed return, while the pair $(r,s)$ records $r$ completed returns with total elapsed time $s$.

The abstract hypotheses below will be stated on a Banach lattice $\B\subset L^1(m_Y)$ of real-valued integrable functions on $Y$, such as $BV(Y)$ in the interval applications in Section~\ref{sec:mp}. A continuous linear functional $\ell:\B\to\R$ is called positive if $v\ge0$ implies $\ell(v)\ge0$. All initial measures considered below are absolutely continuous with respect to $m_Y$, with densities in $\B$. Until Hypothesis~\ref{hyp:renewal} is imposed, the following operators and functionals should be read as formal objects. For generalized Pomeau--Manneville maps, the corresponding requirements are isolated as regular induced-hole conditions in Section~\ref{sec:mp} and then used to verify the abstract hypotheses.

For $n\ge1$, define the return operator $\A_n$ by
\begin{equation}\label{eq:A-n}
	\A_n v\coloneqq\mathcal L_F(\1_{Y_n^{\mathrm{surv}}}v).
\end{equation}
Thus $\A_n v$ is the density on $Y$ obtained by pushing forward $\1_{Y_n^{\mathrm{surv}}}v\,dm_Y$ under $F$. The killed induced transfer operator is
\begin{equation}\label{eq:A}
	\A\coloneqq\sum_{n\ge1}\A_n,
\end{equation}
with convergence on $\B$ imposed in Hypothesis~\ref{hyp:renewal}. Under that hypothesis, $\A$ is a bounded operator on $\B$, so its powers are defined on $\B$. Thus, for an initial nonnegative density $v$ on $Y$, $\A v$ is the density of the pushforward survivor measure after one completed survivor return to $Y$, and $\A^r v$ is the corresponding density after $r$ completed survivor returns.

For $r,s\ge0$, define $\A_{r,s}$ by
\begin{equation}\label{eq:A-rs}
	\A_{r,s}
	\coloneqq
	\sum_{\substack{n_1+\cdots+n_r=s\\ n_i\ge1}}
	\A_{n_r}\cdots\A_{n_1},
\end{equation}
with the convention $\A_{0,0}=I$ and $\A_{0,s}=0$ for $s\ne0$. The operator $\A_{r,s}$ corresponds to $r$ completed survivor returns whose return times sum to $s$.

Summing over all possible numbers of completed survivor returns, define $\Renew_s$ by
\begin{equation}\label{eq:Renew-s}
	\Renew_s\coloneqq\sum_{r\ge0}\A_{r,s}.
\end{equation}

\subsubsection{Tail functionals and survivor masses}

It remains to attach the final unfinished excursion. This is done by a scalar tail functional, which measures the mass of points that survive without returning to the base for the remaining time.

For $y$ with $\tau_H(y)>t$, let $N_t(y)$ denote the number of completed survivor returns to $Y$ made by the orbit of $y$ up to time $t$.

For an initial density $v\in\B$, let $\mu_v$ denote the finite signed measure on $Y$ defined by
\begin{equation}\label{eq:mu-v}
	\mu_v(A)\coloneqq\int_A v\,dm_Y.
\end{equation}
For $v\ge0$, conditional probabilities and expectations are computed using $\mu_v$; normalizing $v$ is unnecessary because only ratios of survivor masses are used.

Define the tail functional $\Tail_n$ by
\begin{equation}\label{eq:Tail}
	\Tail_n(v)\coloneqq\int_{Y_{>n}^{\mathrm{surv}}}v\,dm_Y,
	\qquad n\ge0,\quad v\in\B.
\end{equation}

Let
\begin{equation}\label{eq:M-r-def}
	M_r(t;v)
	\coloneqq
	\mu_v(\{\tau_H>t,\ N_t=r\})
\end{equation}
be the $\mu_v$-mass of points that survive up to time $t$ and make exactly $r$ completed survivor returns. If the $r$ completed survivor returns have lengths $n_1,\dots,n_r$ and total length $s=n_1+\cdots+n_r$, then $\A_{n_r}\cdots\A_{n_1}v$ is the density on $Y$ obtained by restricting to the corresponding survivor-return itinerary and pushing forward by $F^r$. To have exactly $r$ completed survivor returns by time $t$, the remaining segment of length $t-s$ must lie in $Y_{>t-s}^{\mathrm{surv}}$. Summing first over all such $r$-tuples with total length $s$, and then over $0\le s\le t$, gives
\begin{equation}\label{eq:M-r}
	M_r(t;v)
	=
	\sum_{s=0}^{t}\Tail_{t-s}(\A_{r,s}v).
\end{equation}
Let
\begin{equation}\label{eq:M-total-def}
	M(t;v)
	\coloneqq
	\mu_v(\{\tau_H>t\})
\end{equation}
be the total survival mass. Then
\begin{equation}\label{eq:M-total}
	M(t;v)=\sum_{r\ge0}M_r(t;v)
	=\sum_{s=0}^{t}\Tail_{t-s}(\Renew_s v).
\end{equation}
The second equality follows by summing over the possible numbers of completed survivor returns before the final unfinished excursion.

Thus $\A_{r,s}$ fixes both the number $r$ of completed survivor returns and their total elapsed time $s$, while $\Renew_s$ sums over all possible completed-return counts with total elapsed time $s$. The tail functional $\Tail_{t-s}$ then accounts for the final unfinished excursion. This bookkeeping is the only renewal structure used in the sequel.

\subsection{Renewal hypotheses}

We now isolate the assumptions that drive the abstract renewal theorem. They separate the renewal argument itself from the model-specific estimates verified later for Pomeau--Manneville maps.

The following hypotheses are stated at the induced level.

\begin{hypothesis}[Survivor-renewal hypotheses]\label{hyp:renewal}
	There are a Banach lattice $\B\subset L^1(m_Y)$, a number $\kappa>0$, a positive sequence $(a_t)_{t\ge0}$ with $a_0=1$, $a_t\asymp (1+t)^{-\kappa}$, and $a_{t-s}/a_t\to1$ for each fixed $s\ge0$, and a nonzero positive continuous functional $\ell:\B\to\R$ such that:
	\begin{subconditions}
		\item\label{R-tail} \textbf{Final-tail asymptotic.}
		For every $v\in\B$,
		\begin{equation*}
			\Tail_t(v)=a_t\ell(v)+o(a_t).
		\end{equation*}
		Moreover, for some $C_T<\infty$,
		\begin{equation*}
			\abs{\Tail_t(v)}\le C_T a_t\norm{v}_{\B}
		\end{equation*}
		for all $t\ge0$ and all $v\in\B$.

		\item\label{R-operators} \textbf{Killed induced operator.}
		For each $n\ge1$, $\A_n$ defines a bounded positive operator $\B\to\B$, the series $\A=\sum_{n\ge1}\A_n$ converges in operator norm, and hence $\A:\B\to\B$ is bounded. Moreover,
		\begin{equation*}
			r(\A)<1.
		\end{equation*}
		Consequently
		\begin{equation*}
			\Renew\coloneqq\sum_{r\ge0}\A^r=(I-\A)^{-1}
		\end{equation*}
		exists on $\B$.

		\item\label{R-fixed-local} \textbf{Fixed-return local bound.}
		For each fixed $r\ge1$ there is $C_r<\infty$ such that
		\begin{equation*}
			\norm{\A_{r,s}}_{\B\to\B}
			\le C_r(1+s)^{-1-\kappa}
			\qquad(s\ge0).
		\end{equation*}

		\item\label{R-renewal-local} \textbf{Renewal local bound.}
		The operators $\Renew_s$ are bounded on $\B$, the series $\sum_{s\ge0}\Renew_s$ converges in operator norm and equals $\Renew$. Moreover, there is a constant $C_R<\infty$ such that
		\begin{equation*}
			\norm{\Renew_s}_{\B\to\B}
			\le C_R(1+s)^{-1-\kappa}
			\qquad(s\ge0).
		\end{equation*}
	\end{subconditions}
\end{hypothesis}

\begin{remark}
	The local bounds reflect the distinction between a local return-time estimate and a tail estimate. A completed return of length $s$ has local order $s^{-1-\kappa}$, whereas the event $\{R>t\}$ has tail order $a_t\asymp t^{-\kappa}$. This distinction is what makes the fixed-$r$ asymptotic stable under summation over completed return times.
\end{remark}

The hypotheses above are stated directly in operator form. In applications it is often easier to dominate the return operators by a scalar defective renewal sequence; the next lemma records that reduction.

\begin{lemma}[Scalar sufficient condition for operator renewal domination]\label{lem:scalar-domination}
	Suppose each $\A_n$ is positive on $\B$, and suppose there is a nonnegative sequence $(b_n)_{n\ge1}$ and $C>0$ such that
	\begin{equation}\label{eq:scalar-domination-hyp}
		\norm{\A_n}_{\B\to\B}\le b_n,
		\qquad
		\sum_{k\ge1}b_k<1,
		\qquad
		b_n\le C n^{-1-\kappa}
	\end{equation}
	for all $n\ge1$. Then hypotheses \textup{(\ref{R-operators})}, \textup{(\ref{R-fixed-local})}, and \textup{(\ref{R-renewal-local})} hold.
\end{lemma}

\begin{proof}
	The operator series defining $\A$ converges in norm and
	\begin{equation*}
		\norm{\A}\le \sum_{n\ge1}b_n<1.
	\end{equation*}
	Hence $r(\A)<1$, and the Neumann series
	\begin{equation*}
		\Renew=\sum_{r\ge0}\A^r=(I-\A)^{-1}
	\end{equation*}
	converges in operator norm. This proves the killed induced operator condition \textup{(\ref{R-operators})}.

	For fixed $r\ge1$, the definition of $\A_{r,s}$ gives
	\begin{equation*}
		\norm{\A_{r,s}}_{\B\to\B}
		\le
		\sum_{\substack{n_1+\cdots+n_r=s\\ n_i\ge1}}
		b_{n_1}\cdots b_{n_r}.
	\end{equation*}
	Since $b_n=O(n^{-1-\kappa})$ and $\sum_n b_n<\infty$, the standard fixed-convolution estimate gives
	\begin{equation*}
		\sum_{\substack{n_1+\cdots+n_r=s\\ n_i\ge1}}
		b_{n_1}\cdots b_{n_r}
		\le C_r(1+s)^{-1-\kappa}.
	\end{equation*}
	For $r=2$ this follows by splitting the sum into $n_1\le s/2$ and $n_1>s/2$; the general fixed-$r$ case follows by induction. Thus the fixed-return local bound \textup{(\ref{R-fixed-local})} holds.

	Let $u_0=1$ and
	\begin{equation*}
		u_s=
		\sum_{r\ge1}
		\sum_{\substack{n_1+\cdots+n_r=s\\ n_i\ge1}}
		b_{n_1}\cdots b_{n_r},
		\qquad s\ge1.
	\end{equation*}
	Then $\norm{\Renew_s}\le u_s$. Since $\sum_{n\ge1}b_n<1$ and $b_n=O(n^{-1-\kappa})$, the standard defective renewal estimate for locally subexponential tails gives
	\begin{equation*}
		u_s=O(s^{-1-\kappa});
	\end{equation*}
	see, for example,~\cite{FossKorshunovZachary2013}. Moreover, $\sum_s u_s<\infty$, so the series $\sum_{s\ge0}\Renew_s$ converges in operator norm. Expanding by completed-return number gives
	\begin{equation*}
		\sum_{s\ge0}\Renew_s
		=\sum_{r\ge0}\sum_{s\ge0}\A_{r,s}
		=\sum_{r\ge0}\A^r
		=\Renew.
	\end{equation*}
	Therefore the renewal local bound \textup{(\ref{R-renewal-local})} holds.
\end{proof}

\begin{remark}
	Lemma~\ref{lem:scalar-domination} is only an operator-side criterion. It does not verify the final-tail asymptotic or final-tail bound in \textup{(\ref{R-tail})}; those estimates depend on the geometry of the final survivor tail set and must be checked separately.
\end{remark}

\section{The survivor-conditioned renewal theorem}\label{sec:renewal}

We first prove the fixed-return-count asymptotic. The point is that, at leading order, the long observation time is carried by the final unfinished excursion; any fixed number of completed survivor returns contributes only through the killed induced operator $\A$.

\begin{theorem}[Fixed-$r$ survivor asymptotic]\label{thm:fixed-r}
	Assume Hypothesis~\ref{hyp:renewal}. Let $v\in\B$. For every fixed $r\ge0$,
	\begin{equation}\label{eq:fixed-r-asymp}
		M_r(t;v)
		=
		\mu_v(\{\tau_H>t,\ N_t=r\})
		=
		a_t\ell(\A^r v)+o(a_t)
		\qquad(t\to\infty).
	\end{equation}
\end{theorem}

\begin{proof}
	For $r=0$, \eqref{eq:M-r} gives $M_0(t;v)=\Tail_t(v)$, so the claim is exactly the final-tail asymptotic in \textup{(\ref{R-tail})}. Assume $r\ge1$.

	Fix $h\ge0$. For $0\le s\le h$, the fixed-shift regularity of $a_t$ gives $a_{t-s}/a_t\to1$, while the final-tail asymptotic in \textup{(\ref{R-tail})} gives
	\begin{equation*}
		\frac{\Tail_{t-s}(\A_{r,s}v)}{a_t}
		\longrightarrow
		\ell(\A_{r,s}v).
	\end{equation*}
	Since the sum over $s\le h$ is finite,
	\begin{equation}\label{eq:finite-part}
		\lim_{t\to\infty}
		\frac1{a_t}\sum_{s=0}^{h}\Tail_{t-s}(\A_{r,s}v)
		=
		\sum_{s=0}^{h}\ell(\A_{r,s}v).
	\end{equation}
	It remains to show that the contribution from $s>h$ is negligible after $h\to\infty$.

	Split
	\begin{equation*}
		\sum_{h<s\le t}=
		\sum_{h<s\le t/2}+
		\sum_{t/2<s\le t}.
	\end{equation*}
	For $h<s\le t/2$, polynomial comparability gives a constant $C_a$ such that $a_{t-s}\le C_a a_t$. Hence the final-tail bound in \textup{(\ref{R-tail})} gives, uniformly for $0\le s\le t/2$,
	\begin{equation*}
		\abs{\Tail_{t-s}(\A_{r,s}v)}
		\le C_T C_a a_t\norm{\A_{r,s}v}_{\B}.
	\end{equation*}
	Therefore
	\begin{equation*}
		\frac1{a_t}
		\sum_{h<s\le t/2}\abs{\Tail_{t-s}(\A_{r,s}v)}
		\le
		C\norm{v}_{\B}\sum_{s>h}(1+s)^{-1-\kappa},
	\end{equation*}
	which tends to zero as $h\to\infty$.

	For $t/2<s\le t$, set $u=t-s$. By the final-tail bound in \textup{(\ref{R-tail})} and the fixed-return local bound in \textup{(\ref{R-fixed-local})},
	\begin{equation*}
		\begin{aligned}
			\sum_{t/2<s\le t}\abs{\Tail_{t-s}(\A_{r,s}v)}
			 & \le C_T C_r\norm{v}_{\B}
			\sum_{u=0}^{t/2} a_u(1+t-u)^{-1-\kappa} \\
			 & \le C\norm{v}_{\B} t^{-1-\kappa}
			\sum_{u=0}^{t/2} a_u.
		\end{aligned}
	\end{equation*}
	Since $a_u\asymp (1+u)^{-\kappa}$ for $u\ge1$ and $a_0=1$, this last expression is $o(t^{-\kappa})=o(a_t)$ for every $\kappa>0$. Thus
	\begin{equation}\label{eq:tail-negligible}
		\lim_{h\to\infty}\limsup_{t\to\infty}
		\frac1{a_t}
		\sum_{s>h}\abs{\Tail_{t-s}(\A_{r,s}v)}=0.
	\end{equation}
	Combining \eqref{eq:finite-part} and \eqref{eq:tail-negligible},
	\begin{equation*}
		\lim_{t\to\infty}\frac{M_r(t;v)}{a_t}
		=
		\sum_{s\ge0}\ell(\A_{r,s}v).
	\end{equation*}
	The series $\sum_s\A_{r,s}$ converges in operator norm to $\A^r$, and $\ell$ is continuous. Hence the last sum is $\ell(\A^r v)$.
\end{proof}

We now pass from fixed return count to the full survivor-conditioned law by summing over all possible numbers of completed survivor returns. The renewal local bound \textup{(\ref{R-renewal-local})} is what keeps this summation on the same tail scale $a_t$.

\begin{theorem}[Survivor-conditioned renewal law]\label{thm:renewal-law}
	Assume Hypothesis~\ref{hyp:renewal}. Let $v\in\B$ be nonnegative and suppose
	\begin{equation}\label{eq:positive-denominator}
		\ell(\Renew v)>0.
	\end{equation}
	Then the total survival mass satisfies
	\begin{equation}\label{eq:survival-asymp-expanded}
		M(t;v)=\mu_v(\{\tau_H>t\})=a_t\ell(\Renew v)+o(a_t),
	\end{equation}
	and consequently
	\begin{equation}\label{eq:survival-asymp}
		M(t;v)\sim a_t\ell(\Renew v).
	\end{equation}
	Moreover, for every fixed $r\ge0$,
	\begin{equation}\label{eq:Nt-limit}
		\mathbb P_v(N_t=r\mid\tau_H>t)
		\longrightarrow
		\pi_r
		\coloneqq
		\frac{\ell(\A^r v)}{\ell(\Renew v)}.
	\end{equation}
	The sequence $(\pi_r)_{r\ge0}$ defines a probability distribution on $\{0,1,2,\dots\}$.
\end{theorem}

\begin{proof}
	Using \eqref{eq:M-total},
	\begin{equation*}
		M(t;v)=\sum_{s=0}^{t}\Tail_{t-s}(\Renew_s v).
	\end{equation*}
	The proof of Theorem~\ref{thm:fixed-r} used only the local bound for $\A_{r,s}$ and the operator-norm identity $\sum_s\A_{r,s}=\A^r$. By \textup{(\ref{R-renewal-local})}, the family $(\Renew_s)$ satisfies the analogous local bound and $\sum_s\Renew_s=\Renew$ in operator norm. Thus we do not sum the fixed-$r$ limits over $r$; instead, the same argument is applied after the completed-return counts have already been summed at the operator level. This gives
	\begin{equation*}
		M(t;v)=a_t\sum_{s\ge0}\ell(\Renew_s v)+o(a_t)
		=a_t\ell(\Renew v)+o(a_t),
	\end{equation*}
	which proves \eqref{eq:survival-asymp-expanded}. Since $\ell(\Renew v)>0$, the asymptotic equivalence \eqref{eq:survival-asymp} follows from \eqref{eq:survival-asymp-expanded}. The fixed-$r$ limit \eqref{eq:Nt-limit} follows by dividing Theorem~\ref{thm:fixed-r} by \eqref{eq:survival-asymp}. Finally,
	\begin{equation*}
		\sum_{r\ge0}\pi_r
		=
		\frac{\ell(\sum_{r\ge0}\A^r v)}{\ell(\Renew v)}=1.\qedhere
	\end{equation*}
\end{proof}

When the initial density is conditionally invariant for the killed induced system, the abstract limiting distribution becomes completely explicit. The following corollary is the deterministic analogue of the geometric killing law in the stochastic model.

\begin{corollary}[Geometric law for a conditionally invariant density]\label{cor:geometric}
	Assume Hypothesis~\ref{hyp:renewal}. Suppose that $h\in\B$ satisfies
	\begin{equation*}
		h\ge0,
		\qquad \int_Y h\,dm_Y=1,
		\qquad \A h=\lambda h,
		\qquad 0<\lambda<1,
		\qquad \ell(h)>0.
	\end{equation*}
	Then $h$ is conditionally invariant for the killed induced operator, and
	\begin{equation*}
		\mathbb P_h(N_t=r\mid\tau_H>t)
		\longrightarrow
		(1-\lambda)\lambda^r,
		\qquad r=0,1,2,\dots.
	\end{equation*}
\end{corollary}

\begin{proof}
	Since $\A^r h =\lambda^r h$ and $\Renew h=(1-\lambda)^{-1}h$, Theorem~\ref{thm:renewal-law} gives
	\begin{equation*}
		\pi_r=
		\frac{\lambda^r\ell(h)}{(1-\lambda)^{-1}\ell(h)}
		=(1-\lambda)\lambda^r.\qedhere
	\end{equation*}
\end{proof}

\section{General observables and renewal rewards}\label{sec:observables}

We next add an observable to the survivor-renewal scheme. The completed survivor returns carry induced rewards, while the final unfinished excursion is handled by a separate rewarded tail functional.

For an observable $\psi:X\to\R$, write
\begin{equation}\label{eq:birkhoff-sum}
	S_n\psi\coloneqq\sum_{k=0}^{n-1}\psi\circ f^k.
\end{equation}
The reward accumulated during one return excursion is
\begin{equation}\label{eq:Psi-R}
	\Psi_R(y)\coloneqq S_{R(y)}\psi(y),
	\qquad R(y)<\infty.
\end{equation}
For $n\ge1$, define the completed-reward operator $\C_{\psi,n}$ by
\begin{equation}\label{eq:C-n}
	\C_{\psi,n}v
	\coloneqq
	\mathcal L_F(\1_{Y_n^{\mathrm{surv}}}\Psi_R v).
\end{equation}
When the series converges, let
\begin{equation}\label{eq:C-psi}
	\C_\psi\coloneqq\sum_{n\ge1}\C_{\psi,n}.
\end{equation}
Define the final-tail reward functional $\Tail_{\psi,t}$ by
\begin{equation}\label{eq:Tail-psi}
	\Tail_{\psi,t}(v)
	\coloneqq
	\int_{Y_{>t}^{\mathrm{surv}}} S_t\psi\,v\,dm_Y.
\end{equation}

To account for the position of the rewarded excursion within a sequence of survivor returns, define the elapsed-time reward operators $\D_{\psi,s}$ by
\begin{equation}\label{eq:D-s}
	\D_{\psi,s}
	\coloneqq
	\sum_{i+n+j=s}\Renew_i\C_{\psi,n}\Renew_j,
	\qquad s\ge0,
\end{equation}
where $n\ge1$ and $i,j\ge0$. Empty sums are zero, so $\D_{\psi,0}=0$.

The observable theory is organized in two stages: first a boundedness result, and then a convergence result under an additional tail-profile assumption.

\begin{hypothesis}[Observable renewal bounds]\label{hyp:reward-bound}
	Assume Hypothesis~\ref{hyp:renewal}. Let $\psi:X\to\R$ be an observable for which the reward operators and final-tail reward functionals above are defined. There are constants $C_{\psi,T},K_\psi,C_D<\infty$, bounded linear operators
	\[
		\C_{\psi,n}:\B\to\B,
		\qquad n\ge1,
	\]
	and bounded linear operators $\C_\psi,\D_\psi:\B\to\B$ such that the following conditions hold.
	\begin{subconditions}
		\item\label{B-final} \textbf{Final-tail reward bound.}
		For all $t\ge0$ and all $v\in\B$,
		\begin{equation*}
			\abs{\Tail_{\psi,t}(v)}\le C_{\psi,T} a_t\norm{v}_{\B}.
		\end{equation*}

		\item\label{B-completed} \textbf{Completed-reward operators.}
		The series defining $\C_\psi$ converges in operator norm on $\B$:
		\begin{equation*}
			\sum_{n\ge1}\C_{\psi,n}=\C_\psi.
		\end{equation*}
		Moreover,
		\begin{equation*}
			\norm{\C_{\psi,n}}_{\B\to\B}
			\le K_\psi(1+n)^{-1-\kappa}\log(n+2)
			\qquad(n\ge1).
		\end{equation*}

		\item\label{B-D-local} \textbf{Double-renewed reward local bound.}
		The operators $\D_{\psi,s}$ are bounded on $\B$, and the series $\sum_s\D_{\psi,s}$ converges in operator norm. Its sum satisfies
		\begin{equation*}
			\D_\psi=\Renew\C_\psi\Renew.
		\end{equation*}
		Moreover,
		\begin{equation*}
			\norm{\D_{\psi,s}}_{\B\to\B}
			\le C_D(1+s)^{-1-\kappa}\log(s+2)
			\qquad(s\ge1).
		\end{equation*}
	\end{subconditions}
\end{hypothesis}

\begin{remark}
	The logarithmic losses in \textup{(\ref{B-completed})} and \textup{(\ref{B-D-local})} are harmless for the survivor-scale estimates. Indeed, since $a_t\asymp(1+t)^{-\kappa}$, the large-$s$ part of the proof only requires
	\begin{equation}\label{eq:log-negligible}
		t^{-1-\kappa}\log(t+2)\sum_{u=0}^{\floor{t/2}}a_u=o(a_t),
	\end{equation}
	which holds for every $\kappa>0$. The reason for allowing these logarithmic losses is the PM application: a return interval passing near the neutral fixed point can contribute logarithmically to the accumulated observable.
\end{remark}

The next result concerns the survivor-conditioned expectation
\begin{equation*}
	\E_v[S_t\psi\mid\tau_H>t]
	=
	\frac{\int_{\{\tau_H>t\}}S_t\psi\,v\,dm_Y}
	{\mu_v(\{\tau_H>t\})}.
\end{equation*}
Under the domination estimates above, the numerator has the same order as the survival mass $M(t;v)=\mu_v(\{\tau_H>t\})$. This gives boundedness of the conditioned expectation without requiring a limiting profile for the final-tail reward.

\begin{theorem}[Boundedness of survivor-conditioned Birkhoff sums]\label{thm:observable-bounded}
	Assume Hypothesis~\ref{hyp:reward-bound}. Let $v\in\B$ be nonnegative and suppose $\ell(\Renew v)>0$. Then
	\begin{equation}\label{eq:numerator-bound}
		\int_{\{\tau_H>t\}}S_t\psi\,v\,dm_Y
		=
		O(a_t\norm{v}_{\B}),
	\end{equation}
	and consequently
	\begin{equation}\label{eq:observable-bounded}
		\sup_{t\ge1}\abs{\E_v[S_t\psi\mid\tau_H>t]}<\infty.
	\end{equation}
\end{theorem}

\begin{proof}
	Let
	\begin{equation*}
		I_\psi(t;v)\coloneqq\int_{\{\tau_H>t\}}S_t\psi\,v\,dm_Y.
	\end{equation*}
	For $x\in\{\tau_H>t\}$, let $\sigma_t(x)$ be the total length of the completed survivor-return blocks before time $t$. Thus, if the completed survivor-return lengths before time $t$ are $n_1,\dots,n_r$, then
	\begin{equation*}
		\sigma_t(x)=n_1+\cdots+n_r.
	\end{equation*}
	We split the Birkhoff sum into the part accumulated during completed survivor-return blocks and the part accumulated during the final unfinished excursion:
	\begin{equation*}
		S_t\psi(x)
		=
		\sum_{j=0}^{\sigma_t(x)-1}\psi(f^j x)
		+
		\sum_{j=\sigma_t(x)}^{t-1}\psi(f^j x).
	\end{equation*}
	Accordingly, define
	\begin{align}
		F_\psi(t;v)
		 & \coloneqq
		\int_{\{\tau_H>t\}}
		\sum_{j=0}^{\sigma_t(x)-1}\psi(f^j x)\,v\,dm_Y,\label{eq:F-psi-integral} \\
		G_\psi(t;v)
		 & \coloneqq
		\int_{\{\tau_H>t\}}
		\sum_{j=\sigma_t(x)}^{t-1}\psi(f^j x)\,v\,dm_Y.\label{eq:G-psi-integral}
	\end{align}
	Then
	\begin{equation}\label{eq:observable-numerator-split}
		I_\psi(t;v)=F_\psi(t;v)+G_\psi(t;v).
	\end{equation}

	The renewal decomposition gives operator representations for these two terms. The completed-return contribution is
	\begin{equation}\label{eq:F-psi-def}
		F_\psi(t;v)
		=
		\sum_{s=0}^{t}\Tail_{t-s}(\D_{\psi,s}v).
	\end{equation}
	Here $\D_{\psi,s}v$ collects all survivor-return itineraries whose total completed-return length is $s$, with one completed return weighted by its $\psi$-reward. The remaining segment has length $t-s$, so the ordinary tail functional $\Tail_{t-s}$ checks survival through the final unfinished excursion.

	Similarly, the final-excursion contribution is
	\begin{equation}\label{eq:G-psi-def}
		G_\psi(t;v)
		=
		\sum_{s=0}^{t}\Tail_{\psi,t-s}(\Renew_s v).
	\end{equation}
	Here $\Renew_s v$ is the density obtained after summing over all survivor-return itineraries of total completed-return length $s$, and $\Tail_{\psi,t-s}$ records the $\psi$-reward accumulated during the final unfinished excursion.

	We first estimate $F_\psi(t;v)$. For $s\le t/2$, the final-tail bound in \textup{(\ref{R-tail})}, polynomial comparability of $a_t$, and \textup{(\ref{B-D-local})} give
	\begin{equation*}
		\abs{\Tail_{t-s}(\D_{\psi,s}v)}
		\le
		Ca_t(1+s)^{-1-\kappa}\log(s+2)\norm{v}_{\B}.
	\end{equation*}
	The sequence on the right is summable in $s$, so the contribution from $s\le t/2$ is $O(a_t\norm{v}_{\B})$. For $s>t/2$, using \textup{(\ref{R-tail})} and \textup{(\ref{B-D-local})} again gives
	\begin{equation*}
		\sum_{t/2<s\le t}\abs{\Tail_{t-s}(\D_{\psi,s}v)}
		\le
		C\norm{v}_{\B}\,
		t^{-1-\kappa}\log(t+2)
		\sum_{u=0}^{\floor{t/2}}a_u
		=
		o(a_t\norm{v}_{\B}),
	\end{equation*}
	by \eqref{eq:log-negligible}. Hence
	\begin{equation}\label{eq:completed-reward-bound}
		F_\psi(t;v)=O(a_t\norm{v}_{\B}).
	\end{equation}

	We next estimate $G_\psi(t;v)$. For $s\le t/2$, the final-tail reward bound in \textup{(\ref{B-final})} and polynomial comparability of $a_t$ give
	\begin{equation*}
		\abs{\Tail_{\psi,t-s}(\Renew_s v)}
		\le
		Ca_t\norm{\Renew_s v}_{\B}.
	\end{equation*}
	Since $\sum_s\Renew_s$ converges in operator norm by \textup{(\ref{R-renewal-local})}, the contribution from $s\le t/2$ is $O(a_t\norm{v}_{\B})$. For $s>t/2$, the same final-tail reward bound and \textup{(\ref{R-renewal-local})} give, after writing $u=t-s$,
	\begin{equation*}
		\sum_{t/2<s\le t}\abs{\Tail_{\psi,t-s}(\Renew_s v)}
		\le
		C\norm{v}_{\B}
		\sum_{u=0}^{\floor{t/2}}a_u(1+t-u)^{-1-\kappa}
		=
		o(a_t\norm{v}_{\B}).
	\end{equation*}
	Therefore
	\begin{equation}\label{eq:final-reward-bound}
		G_\psi(t;v)=O(a_t\norm{v}_{\B}).
	\end{equation}

	Combining \eqref{eq:observable-numerator-split}, \eqref{eq:completed-reward-bound}, and \eqref{eq:final-reward-bound} gives \eqref{eq:numerator-bound}. The survival asymptotic \eqref{eq:survival-asymp} and the assumption $\ell(\Renew v)>0$ then give \eqref{eq:observable-bounded}.
\end{proof}

For nonnegative observables, Theorem~\ref{thm:observable-bounded} also controls shorter partial sums.

\begin{corollary}[Nonnegative two-parameter form]\label{cor:two-parameter}
	Under the hypotheses of Theorem~\ref{thm:observable-bounded}, if $\psi\ge0$, then
	\begin{equation*}
		\sup_{t\ge1}\sup_{0\le n\le t}
		\E_v[S_n\psi\mid\tau_H>t]<\infty.
	\end{equation*}
\end{corollary}

\begin{proof}
	For $0\le n\le t$, $S_n\psi\le S_t\psi$ on $\{\tau_H>t\}$. The claim follows from Theorem~\ref{thm:observable-bounded}.
\end{proof}

The preceding theorem gives only the uniform bound
\begin{equation*}
	\E_v[S_t\psi\mid\tau_H>t]=O(1).
\end{equation*}
It does not identify the limit, if any, of these survivor-conditioned expectations as $t\to\infty$. The completed-return contribution can be handled by the operator estimates in Hypothesis~\ref{hyp:reward-bound}; the remaining issue is the reward accumulated during the final unfinished excursion. To obtain convergence, that final-tail reward must have a normalized asymptotic on the same scale as the survival mass.

\begin{hypothesis}[Convergent reward hypotheses]\label{hyp:reward}
	Assume Hypothesis~\ref{hyp:reward-bound}. In addition, there is a continuous functional $\ell_\psi:\B\to\R$ such that, for every $v\in\B$,
	\begin{equation}\label{eq:reward-asymptotic-hyp}
		\Tail_{\psi,t}(v)=a_t\ell_\psi(v)+o(a_t).
	\end{equation}
\end{hypothesis}

With this extra asymptotic, the two reward contributions have separate limits: one from the completed survivor-return blocks and one from the final unfinished excursion.

\begin{theorem}[Limit of survivor-conditioned Birkhoff sums]\label{thm:observable-limit}
	Assume Hypothesis~\ref{hyp:reward}. Let $v\in\B$ be nonnegative and suppose $\ell(\Renew v)>0$. Then
	\begin{equation}\label{eq:numerator-asymp}
		\int_{\{\tau_H>t\}}S_t\psi\,v\,dm_Y
		=
		a_t\left[\ell_\psi(\Renew v)+\ell(\D_\psi v)\right]+o(a_t).
	\end{equation}
	Consequently,
	\begin{equation}\label{eq:observable-limit}
		\E_v[S_t\psi\mid\tau_H>t]
		\longrightarrow
		\frac{\ell_\psi(\Renew v)+\ell(\D_\psi v)}{\ell(\Renew v)}.
	\end{equation}
\end{theorem}

\begin{proof}
	Hypothesis~\ref{hyp:reward} includes Hypothesis~\ref{hyp:reward-bound}, so the decomposition from the proof of Theorem~\ref{thm:observable-bounded} applies. Write
	\begin{equation*}
		I_\psi(t;v)
		\coloneqq
		\int_{\{\tau_H>t\}}S_t\psi\,v\,dm_Y
		=
		F_\psi(t;v)+G_\psi(t;v),
	\end{equation*}
	where
	\begin{equation*}
		F_\psi(t;v)=\sum_{s=0}^{t}\Tail_{t-s}(\D_{\psi,s}v)
	\end{equation*}
	is the completed-return contribution, and
	\begin{equation*}
		G_\psi(t;v)=\sum_{s=0}^{t}\Tail_{\psi,t-s}(\Renew_s v)
	\end{equation*}
	is the contribution from the final unfinished excursion.

	We first treat the completed-return contribution. Fix $h\ge0$. For each fixed $s\le h$, the final-tail asymptotic in \textup{(\ref{R-tail})} and the fixed-shift relation $a_{t-s}/a_t\to1$ give
	\begin{equation*}
		\frac{1}{a_t}\Tail_{t-s}(\D_{\psi,s}v)
		\longrightarrow
		\ell(\D_{\psi,s}v),
		\qquad t\to\infty.
	\end{equation*}
	Thus the finite sums converge:
	\begin{equation*}
		\frac{1}{a_t}\sum_{s=0}^{h}\Tail_{t-s}(\D_{\psi,s}v)
		\longrightarrow
		\sum_{s=0}^{h}\ell(\D_{\psi,s}v).
	\end{equation*}
	The error from $s>h$ is controlled exactly as in the proof of Theorem~\ref{thm:observable-bounded}. Splitting at $t/2$ and using \textup{(\ref{R-tail})} and \textup{(\ref{B-D-local})} gives
	\begin{equation*}
		\limsup_{t\to\infty}\frac1{a_t}
		\sum_{h<s\le t}
		\abs{\Tail_{t-s}(\D_{\psi,s}v)}
		\le
		C\norm{v}_{\B}\sum_{s>h}(1+s)^{-1-\kappa}\log(s+2).
	\end{equation*}
	The right-hand side tends to zero as $h\to\infty$. Since $\sum_s\D_{\psi,s}=\D_\psi$ in operator norm and $\ell$ is continuous, we obtain
	\begin{equation}\label{eq:completed-reward}
		\frac{F_\psi(t;v)}{a_t}
		\longrightarrow
		\ell(\D_\psi v).
	\end{equation}

	We next treat the final unfinished excursion. For each fixed $s\le h$, the final-tail reward asymptotic \eqref{eq:reward-asymptotic-hyp} and the fixed-shift relation $a_{t-s}/a_t\to1$ give
	\begin{equation*}
		\frac{1}{a_t}\Tail_{\psi,t-s}(\Renew_s v)
		\longrightarrow
		\ell_\psi(\Renew_s v),
		\qquad t\to\infty.
	\end{equation*}
	Hence
	\begin{equation*}
		\frac{1}{a_t}\sum_{s=0}^{h}\Tail_{\psi,t-s}(\Renew_s v)
		\longrightarrow
		\sum_{s=0}^{h}\ell_\psi(\Renew_s v).
	\end{equation*}
	For the tail $s>h$, the estimates used for $G_\psi$ in Theorem~\ref{thm:observable-bounded}, now with \textup{(\ref{B-final})} and \textup{(\ref{R-renewal-local})}, give
	\begin{equation*}
		\limsup_{t\to\infty}\frac1{a_t}
		\sum_{h<s\le t}
		\abs{\Tail_{\psi,t-s}(\Renew_s v)}
		\le
		C\norm{v}_{\B}\sum_{s>h}(1+s)^{-1-\kappa},
	\end{equation*}
	which tends to zero as $h\to\infty$. Since $\sum_s\Renew_s=\Renew$ in operator norm and $\ell_\psi$ is continuous,
	\begin{equation}\label{eq:final-reward}
		\frac{G_\psi(t;v)}{a_t}
		\longrightarrow
		\ell_\psi(\Renew v).
	\end{equation}

	Combining \eqref{eq:completed-reward} and \eqref{eq:final-reward} gives \eqref{eq:numerator-asymp}. The survival asymptotic \eqref{eq:survival-asymp} and the assumption $\ell(\Renew v)>0$ give \eqref{eq:observable-limit}.
\end{proof}

\section{Entropy of survivor return-length names}\label{sec:entropy}

We use $\Sh$ for Shannon entropy of countable probability vectors: if $p=(p_i)$, then
\begin{equation}\label{eq:shannon-entropy}
	\Sh(p)\coloneqq -\sum_i p_i\log p_i,
\end{equation}
with the convention $0\log0=0$. If $v\in\B$ is nonnegative and $\mu_v(A)>0$, then $\Sh_v(\mathcal P\given A)$ denotes the Shannon entropy of the atom probabilities of $\mathcal P$ under the conditional measure $\mu_v(\cdot\mid A)$.

We now record a symbolic consequence for the return-length names of survivor orbits. If one instead codes by individual survivor-return branches, the entropy hypothesis below must be imposed at the branch level rather than only at the return-length level.

\begin{hypothesis}[Entropy hypotheses]\label{hyp:entropy}
	Assume Hypothesis~\ref{hyp:renewal}. The return-length entropy estimates satisfy the following conditions.
	\begin{subconditions}
		\item\label{E-one-step-domination} \textbf{Uniform one-step entropy domination.}
		There is a probability vector $q=(q_n)_{n\ge1}$ with $\Sh(q)<\infty$ and a constant $C_E<\infty$ such that every one-step survivor-return length distribution which appears before the final tail in the renewal decomposition is dominated by $C_Eq$. Equivalently, for each normalized nonnegative density $w\in\B$ occurring as a conditional density at the start of a completed survivor return, the conditional probabilities
		\begin{equation*}
			p_n(w)
			\coloneqq
			\frac{\int_{Y_n^{\mathrm{surv}}}w\,dm_Y}{\sum_{k\ge1}\int_{Y_k^{\mathrm{surv}}}w\,dm_Y}
		\end{equation*}
		satisfy $p_n(w)\le C_E q_n$ whenever the denominator is nonzero.

		\item\label{E-return-count-domination} \textbf{Uniform return-count domination.}
		For every nonnegative $v\in\B$ with $\ell(\Renew v)>0$, there are constants $C_v<\infty$ and $0<\theta_v<1$ such that
		\begin{equation*}
			\mathbb P_v(N_t\ge r\mid \tau_H>t)
			\le C_v\theta_v^r
		\end{equation*}
		for all $r\ge0$ and all sufficiently large $t$.
	\end{subconditions}
\end{hypothesis}

\begin{theorem}[Zero entropy rate for survivor return-length names]\label{thm:entropy}
	Assume Hypothesis~\ref{hyp:entropy}. Let $v\in\B$ be nonnegative and suppose $\ell(\Renew v)>0$. For each $t\ge1$, let $\mathcal P_t^{\mathrm{len}}$ be the partition of $\{\tau_H>t\}$ according to the completed survivor-return lengths before time $t$ and the final residual time. Then
	\begin{equation}\label{eq:entropy-log}
		\Sh_v(\mathcal P_t^{\mathrm{len}}\given\tau_H>t)=O(\log t),
	\end{equation}
	and consequently
	\begin{equation}\label{eq:entropy-rate}
		\frac1t\Sh_v(\mathcal P_t^{\mathrm{len}}\given\tau_H>t)\longrightarrow0.
	\end{equation}
\end{theorem}

\begin{proof}
	It is enough to prove \eqref{eq:entropy-log} for all sufficiently large $t$; changing finitely many values does not affect the $O(\log t)$ estimate. By \textup{(\ref{E-return-count-domination})}, the conditional laws of $N_t$ have uniformly exponentially decaying tails. Hence
	\begin{equation*}
		\sup_t \E_v[N_t\mid\tau_H>t]<\infty,
		\qquad
		\sup_t \Sh_v(N_t\given \tau_H>t)<\infty,
	\end{equation*}
	where the suprema are over large $t$.

	Given $N_t=r$, a return-length name consists of $r$ completed survivor-return lengths and one final residual time in $\{0,1,\dots,t\}$. The residual time therefore contributes at most $\log(t+1)$ conditional entropy. The domination in \textup{(\ref{E-one-step-domination})}, together with $\Sh(q)<\infty$, gives a uniform constant $K_E<\infty$ such that each completed return length contributes at most $K_E$ conditional entropy. Applying the entropy chain rule gives
	\begin{equation*}
		\Sh_v(\mathcal P_t^{\mathrm{len}}\given \tau_H>t)
		\le
		\Sh_v(N_t\given \tau_H>t)
		+K_E\E_v[N_t\mid\tau_H>t]
		+\log(t+1).
	\end{equation*}
	The first two terms are uniformly bounded in $t$, so \eqref{eq:entropy-log} follows. Dividing by $t$ gives \eqref{eq:entropy-rate}.
\end{proof}

\section{Pomeau--Manneville maps and regular holes}\label{sec:mp}

We now formulate a concrete intermittent class for which the abstract hypotheses above can be verified, and hence to which the preceding survivor-renewal results apply. The map hypotheses are the generalized Pomeau--Manneville hypotheses used in the author's earlier work; the hole hypotheses isolate the additional open-system regularity needed for these results.

There are two distinct conclusions in the application. The two-sided neutral estimates are sufficient for domination estimates and hence for bounded survivor-conditioned rewards. Convergence of conditioned expectations requires additional asymptotic regularity near the neutral fixed point, because the final-tail reward must have a normalized limit rather than just the right order of magnitude.

\begin{definition}[Generalized Pomeau--Manneville map]\label{def:generalized-mp-map}
	Let $\gamma>0$. We say that $f:[0,1]\to[0,1]$ is a generalized Pomeau--Manneville map with exponent $\gamma$ if $f$ is a two-branched interval map and there exists $r_1\in(0,1)$ such that the following conditions hold, using one-sided derivatives as needed:
	\begin{subconditions}
		\item\label{MP-full-branches} $f((0,r_1))=f((r_1,1))=(0,1)$.

		\item\label{MP-regularity} $f$ is $C^1$ on $[0,r_1)$ and $C^2$ on $(0,r_1)\cup(r_1,1]$.

		\item\label{MP-expansion} $f(0)=f(r_1)=0$, $f'(0)=1$, and $\abs{f'}>1$ on $(0,1]$ away from partition boundaries.

		\item\label{MP-bounded-derivatives} $\sup_{x\in[0,1]}\abs{f'(x)}<\infty$ and $\sup_{x\in(r_1,1]}\abs{f''(x)}<\infty$.

		\item\label{MP-neutral-curvature} There is $C_f\ge1$ such that, for every $x\in(0,r_1)$,
		\begin{equation*}
			C_f^{-1}\le \frac{f''(x)}{x^{\gamma-1}}\le C_f.
		\end{equation*}
	\end{subconditions}
\end{definition}

The next definition is only needed for convergence statements. It rules out oscillatory neutral coefficients which preserve two-sided PM estimates but may destroy final-tail reward limits.

\begin{definition}[Asymptotically regular behavior near the neutral fixed point]\label{def:asymptotically-regular}
	A generalized Pomeau--Manneville map has asymptotically regular behavior near the neutral fixed point if there is $c_f>0$ such that
	\begin{equation}\label{eq:neutral-asymptotic}
		\frac{f''(x)}{x^{\gamma-1}}\longrightarrow c_f
		\qquad(x\to0).
	\end{equation}
\end{definition}

\begin{remark}
	Condition \eqref{eq:neutral-asymptotic} is not needed for the renewal law or for boundedness of Lyapunov stretching. It is used only when one wants convergence of reward averages for observables whose leading contribution comes from the neutral branch. Under \eqref{eq:neutral-asymptotic},
	\begin{equation*}
		f'(x)-1\sim \frac{c_f}{\gamma}x^\gamma,
		\qquad
		\log\abs{f'(x)}\sim \frac{c_f}{\gamma}x^\gamma.
	\end{equation*}
\end{remark}

\begin{remark}
	As in the author's earlier PM articles, the assumption in \textup{(\ref{MP-expansion})} that $\abs{f'_+(r_1)}>1$ is used only to ensure the required expansion and distortion properties of the first return map to $[r_1,1]$. The same conclusions remain valid if $\abs{f'_+(r_1)}=1$, provided the induced first return map is uniformly expanding and satisfies the bounded distortion estimates used below. We state the interval-preimage formulas below for the increasing right-branch case; the decreasing right-branch case requires only the corresponding orientation changes.
\end{remark}

Let $Y=[r_1,1]$, and let $R$ be the first return time to $Y$. Define $r_0=1$, let $r_1$ be the branch point above, and, for $n\ge1$, let $r_{n+1}$ be the unique preimage of $r_n$ in $(0,r_1)$. Also define $p_0=1$, and, for $n\ge1$, let $p_n$ be the unique preimage of $r_n$ in $(r_1,1)$. Thus
\begin{equation}\label{eq:general-mp-rp-sequences}
	f(r_{n+1})=r_n,
	\qquad
	f(p_n)=r_n
	\qquad(n\ge1).
\end{equation}
Then $r_n\searrow0$ and, in the increasing right-branch case, $p_n\searrow r_1$. The standard estimates for this class~\cite{LSV1999,Young1999,Gouezel2004} give
\begin{equation}\label{eq:general-mp-tail-estimates}
	r_n\asymp n^{-1/\gamma},
	\qquad
	r_{n-1}-r_n\asymp n^{-1-1/\gamma},
	\qquad
	p_{n-1}-p_n\asymp n^{-1-1/\gamma}.
\end{equation}
Consequently,
\begin{equation}\label{eq:general-return-tail}
	m_Y(R>n)\asymp n^{-1/\gamma},
	\qquad
	m_Y(R=n)\asymp n^{-1-1/\gamma}.
\end{equation}
The induced map $F=f^R:Y\to Y$ is a full-branch Markov map with uniform expansion and bounded distortion.

We assume the following structure on the hole $H$. The condition $H\subset[\varepsilon,1]$ means only that the hole is bounded away from the neutral fixed point; the hole need not be the interval $[\varepsilon,1]$. This induced-system formulation keeps the renewal theorem separate from the technical problem of classifying all geometric holes for which the required killed-operator and final-tail estimates hold.

\begin{definition}[Regular induced hole]\label{def:regular-induced-hole}
	Let $f$ be a generalized Pomeau--Manneville map and let $Y=[r_1,1]$. A hole $H\subset[\varepsilon,1]$, $\varepsilon>0$, is called regular for the inducing scheme if the following hold for $\B=BV(Y)$.
	\begin{subconditions}
		\item\label{H-survivor-branches} The survivor-return sets $Y_n^{\mathrm{surv}}=\{R=n,\ \tau_H>n\}$ are finite unions of return-cylinder subintervals up to endpoints of zero measure, and multiplication by $\1_{Y_n^{\mathrm{surv}}}$ followed by transfer under $F$ acts boundedly on $BV(Y)$.

		\item\label{H-spectral-radius} The killed induced operator $\A=\sum_n\A_n$ has spectral radius $r(\A)<1$ on $BV(Y)$.

		\item\label{H-local-bounds} The local bounds \textup{(\ref{R-fixed-local})} and \textup{(\ref{R-renewal-local})} hold with $\kappa=1/\gamma$.

		\item\label{H-final-tail} There is a positive sequence $(a_t)_{t\ge0}$ with $a_0=1$, $a_t\asymp(1+t)^{-1/\gamma}$, and $a_{t-s}/a_t\to1$ for each fixed $s\ge0$, and a constant $c_H>0$, such that for every $v\in BV(Y)$,
		\begin{equation*}
			\Tail_t(v)=a_t c_H v(r_1+)+o(a_t).
		\end{equation*}
		Moreover, for some $K_H<\infty$,
		\begin{equation*}
			\abs{\Tail_t(v)}\le K_H a_t\norm{v}_{BV}
		\end{equation*}
		for all $t\ge0$ and all $v\in BV(Y)$.
	\end{subconditions}
\end{definition}

\begin{remark}
	Definition~\ref{def:regular-induced-hole} is an admissibility condition for sharp asymptotics, not a maximal class of holes. It is designed to isolate exactly the induced open-system estimates needed by the abstract renewal theorem. The intended examples are standard finite-interval holes bounded away from the neutral fixed point, subject to the usual endpoint exclusions and visibility conditions. The present formulation keeps the paper focused on the renewal structure rather than on a maximal classification of admissible holes. If a hole has a component inside the neutral branch but is bounded away from $0$, then sufficiently long return intervals crossing that component are excluded from the survivor-return sets, which strengthens the local survivor-return bounds. If the hole lies in the base, then the condition reduces to the standard full-branch Gibbs--Markov map with a regular range hole. The condition $c_H>0$ excludes holes which obstruct the long-excursion entrance region at the base endpoint.
\end{remark}

For convergence of rewards, the mass asymptotic in the regular-hole definition is not enough. The final tail must also have stable asymptotics after being weighted by the normalized return-time profile.

\begin{definition}[Tail-profile regularity]\label{def:tail-profile-regular}
	A regular induced hole is called tail-profile regular if final tails have asymptotics not only for their total mass, but also for logarithmically bounded profiles of the normalized return time. More precisely, for every profile function \(\Phi:(1,\infty)\to\mathbb R\) which is continuous except at finitely many points and satisfies
	\begin{equation*}
		\abs{\Phi(u)}\le C_\Phi\left(1+\log\frac{u}{u-1}\right),
		\qquad u>1,
	\end{equation*}
	there is a constant \(c_H(\Phi)\) such that, for every \(v\in BV(Y)\),
	\begin{equation}\label{eq:tail-profile-regularity}
		\int_{Y_{>t}^{\mathrm{surv}}}\Phi\left(\frac{R(y)}{t}\right)v(y)\,dm_Y(y)
		=
		a_t c_H(\Phi)v(r_1+)+o(a_t).
	\end{equation}
	The same bound is required uniformly: there is a constant \(K_\Phi<\infty\) such that
	\begin{equation*}
		\left\lvert
		\int_{Y_{>t}^{\mathrm{surv}}}\Phi\left(\frac{R(y)}{t}\right)v(y)\,dm_Y(y)
		\right\rvert
		\le K_\Phi a_t\norm{v}_{BV}.
	\end{equation*}
\end{definition}

\begin{remark}
	Tail-profile regularity is stronger than the mass asymptotic in \textup{(\ref{H-final-tail})}. It is the additional condition needed to turn neutral-branch reward bounds into reward limits. For standard interval holes whose sufficiently long final tails are not cut by the hole, it follows from the regular variation of the return-time partition. It is stated separately because a total tail asymptotic alone does not logically force convergence after weighting by the normalized return-time profile.
\end{remark}

With these definitions in place, the abstract renewal hypotheses become a direct consequence of the induced estimates built into the regular-hole condition.

\begin{theorem}[Survivor renewal for generalized PM maps]\label{thm:MP-renewal}
	Let $f$ be a generalized Pomeau--Manneville map with exponent $\gamma>0$, let $Y=[r_1,1]$, and let $H\subset[\varepsilon,1]$ be a regular induced hole. Then Hypothesis~\ref{hyp:renewal} holds on $\B=BV(Y)$ with
	\begin{equation*}
		\kappa=1/\gamma.
	\end{equation*}
	Moreover, the final-tail functional has the form
	\begin{equation}\label{eq:ell-eval}
		\ell(v)=c_H v(r_1+),
	\end{equation}
	where $v(r_1+)$ denotes the right-hand limit of the $BV$ representative of $v$ at the left endpoint $r_1$ of $Y$, and $c_H>0$ is the constant from \textup{(\ref{H-final-tail})}.
\end{theorem}

\begin{proof}
	The expansion and distortion estimates for the induced first return map follow from \textup{(\ref{MP-full-branches})}--\textup{(\ref{MP-neutral-curvature})} and the standard estimates \eqref{eq:general-mp-tail-estimates}. The tail estimates \eqref{eq:general-return-tail} give the local polynomial order $n^{-1-1/\gamma}$ for completed return intervals.

	The regular induced-hole assumptions verify the abstract renewal hypotheses condition by condition. Condition \textup{(\ref{H-survivor-branches})}, together with the bounded distortion of the induced branches, gives boundedness of the survivor-return operators on $BV(Y)$. Condition \textup{(\ref{H-spectral-radius})} gives the killed induced operator and the bound on its spectral radius, hence \textup{(\ref{R-operators})}. Condition \textup{(\ref{H-local-bounds})} gives \textup{(\ref{R-fixed-local})} and \textup{(\ref{R-renewal-local})} with $\kappa=1/\gamma$.

	It remains only to identify the final-tail functional. This is exactly \textup{(\ref{H-final-tail})}, which gives
	\begin{equation*}
		\Tail_t(v)=a_t c_H v(r_1+)+o(a_t)
	\end{equation*}
	and the corresponding uniform tail bound on $BV(Y)$. Thus \textup{(\ref{R-tail})} holds with $\ell$ as in \eqref{eq:ell-eval}. Therefore Hypothesis~\ref{hyp:renewal} holds.
\end{proof}

\subsection{Observable verification}

The next two propositions separate boundedness from convergence. A neutral-branch size and variation condition gives the reward domination hypotheses. A genuine neutral asymptotic gives the final-tail reward limit.

The first technical point is the induced $BV$ estimate for the completed-reward operators. The estimate uses two facts: along a long neutral excursion the reward is at most logarithmic, and bounded distortion converts this logarithmic weight into an operator norm bound.

\begin{lemma}[Induced reward variation bound]\label{lem:induced-reward-variation}
	Let $f$ and $H$ be as in Theorem~\ref{thm:MP-renewal}. Suppose $\psi:[0,1]\to\R$ is bounded, has bounded variation on every interval $[\delta,1]$ with $\delta>0$, and, for some $K_0,\delta_0>0$,
	\begin{equation}\label{eq:psi-vanish}
		\abs{\psi(x)}\le K_0x^\gamma,
		\qquad
		\abs{\psi'(x)}\le K_0x^{\gamma-1}
		\qquad(0<x\le\delta_0),
	\end{equation}
	where the derivative condition may be replaced by the corresponding bounded-variation estimate on the neutral partition intervals. Then there is a constant \(K_\psi<\infty\) such that
	\begin{equation}\label{eq:Cpsi-n-bound}
		\norm{\C_{\psi,n}}_{BV\to BV}
		\le K_\psi n^{-1-1/\gamma}\log(n+2).
	\end{equation}
\end{lemma}

\begin{proof}
	During a return interval of length $n$ which passes through the neutral branch, the neutral-branch coordinates satisfy $f^k y\asymp (n-k+1)^{-1/\gamma}$ when indexed backward from the return. Thus \eqref{eq:psi-vanish} gives
	\begin{equation*}
		\abs{\psi(f^k y)}\le C(n-k+1)^{-1}
	\end{equation*}
	for the portion of the return interval in the neutral branch, while the uniformly expanding pieces away from $0$ consist of uniformly bounded orbit segments and therefore contribute $O(1)$. Hence
	\begin{equation}\label{eq:log-reward-bound}
		\sup_{R=n}\abs{S_R\psi}\le C\log(n+2).
	\end{equation}
	The derivative or variation condition in \eqref{eq:psi-vanish} gives the matching variation estimate
	\begin{equation*}
		\Var_{Y_n}(S_R\psi)\le C\log(n+2)
	\end{equation*}
	on each return cylinder, up to the usual distortion constants. Since the $n$th induced branch has image $Y$, bounded distortion and the local size estimate $m_Y(R=n)=O(n^{-1-1/\gamma})$ give \eqref{eq:Cpsi-n-bound}.
\end{proof}

The domination proposition now follows by combining the induced reward variation bound with the final-tail size estimate.

\begin{proposition}[Reward domination for PM observables]\label{prop:observable-domination}
	Let $f$ and $H$ be as in Theorem~\ref{thm:MP-renewal}. Suppose $\psi$ satisfies the hypotheses of Lemma~\ref{lem:induced-reward-variation}. Then $\psi$ satisfies Hypothesis~\ref{hyp:reward-bound}. Consequently the survivor-conditioned Birkhoff sums are bounded.
\end{proposition}

\begin{proof}
	Lemma~\ref{lem:induced-reward-variation} gives the completed-reward local bound
	\begin{equation*}
		\norm{\C_{\psi,n}}_{BV\to BV}
		\le K_\psi n^{-1-1/\gamma}\log(n+2).
	\end{equation*}
	Since $\sum_n n^{-1-1/\gamma}\log(n+2)<\infty$, the series defining $\C_\psi$ converges in operator norm on $BV(Y)$. This verifies the completed-reward operator condition \textup{(\ref{B-completed})}.

	It remains on the completed-return side to check the double-renewed operators. By \textup{(\ref{R-renewal-local})} and the estimate above,
	\begin{equation*}
		\norm{\D_{\psi,s}}_{BV\to BV}
		\le
		C\sum_{i+n+j=s}
		(1+i)^{-1-1/\gamma}
		n^{-1-1/\gamma}\log(n+2)
		(1+j)^{-1-1/\gamma}.
	\end{equation*}
	The standard locally subexponential convolution estimate gives
	\begin{equation*}
		\norm{\D_{\psi,s}}_{BV\to BV}
		\le C(1+s)^{-1-1/\gamma}\log(s+2).
	\end{equation*}
	The same absolute summability justifies the Cauchy product and gives $\sum_s\D_{\psi,s}=\Renew\C_\psi\Renew$ in operator norm. Thus \textup{(\ref{B-D-local})} holds.

	It remains to record the final-tail reward bound. If $R=N>t$, then the part of the orbit observed before time $t$ satisfies
	\begin{equation*}
		\abs{S_t\psi(y)}
		\le
		C\left(1+\sum_{j=N-t+1}^{N}\frac1j\right)
		\le
		C\left(1+\log\frac{N+1}{N-t+1}\right).
	\end{equation*}
	Using the local estimate $m_Y(R=N)=O(N^{-1-1/\gamma})$ and bounded distortion on the final-tail cylinders, we obtain
	\begin{equation*}
		\abs{\Tail_{\psi,t}(v)}
		\le
		C\norm{v}_{BV}
		\sum_{N>t}N^{-1-1/\gamma}
		\left(1+\log\frac{N+1}{N-t+1}\right).
	\end{equation*}
	The last sum is $O(t^{-1/\gamma})=O(a_t)$, by splitting into $t<N\le2t$ and $N>2t$, or by the change of variables $N=\floor{tu}$. Therefore \textup{(\ref{B-final})} holds. The three conditions of Hypothesis~\ref{hyp:reward-bound} are verified.
\end{proof}

\begin{proposition}[Reward convergence for asymptotically regular PM observables]\label{prop:observable-convergence}
	Let $f$ and $H$ be as in Theorem~\ref{thm:MP-renewal}. Assume that $f$ has asymptotically regular behavior near the neutral fixed point in the sense of Definition~\ref{def:asymptotically-regular}, and that $H$ is tail-profile regular in the sense of Definition~\ref{def:tail-profile-regular}. Suppose $\psi$ satisfies the hypotheses of Proposition~\ref{prop:observable-domination} and, in addition, there is $c_\psi\in\R$ such that
	\begin{equation}\label{eq:psi-asymptotic}
		\frac{\psi(x)}{x^\gamma}\to c_\psi
		\qquad(x\to0).
	\end{equation}
	Then $\psi$ satisfies Hypothesis~\ref{hyp:reward}. Consequently the survivor-conditioned Birkhoff sums converge to a finite limit.
\end{proposition}

\begin{proof}
	By Proposition~\ref{prop:observable-domination}, Hypothesis~\ref{hyp:reward-bound} holds. It remains to prove the final-tail reward asymptotic. The asymptotic regularity near the neutral fixed point gives regular variation of the neutral preimage sequence; in particular $r_n^\gamma\sim c_*n^{-1}$ for some $c_*>0$. For fixed $u>1$ and $N=\floor{tu}$, \eqref{eq:psi-asymptotic} gives
	\begin{equation*}
		\sum_{j=N-t+1}^{N}\psi(r_j)
		\longrightarrow
		c_\psi c_*\log\frac{u}{u-1}.
	\end{equation*}
	The entrance contribution from the base converges to the right-hand limit $\psi(r_1+)$. Thus the limiting final-tail reward profile is
	\begin{equation}\label{eq:psi-tail-profile}
		\Phi_\psi(u)
		=
		\psi(r_1+)+c_\psi c_*\log\frac{u}{u-1},
		\qquad u>1.
	\end{equation}

	The same bounded-distortion and regular-variation estimates give the corresponding replacement on final-tail cylinders. On the part of $Y_{>t}^{\mathrm{surv}}$ with return time $R=N>t$, the final-tail reward $S_t\psi$ differs from $\Phi_\psi(N/t)$ by an error whose integral against any $BV$ density is $o(a_t)\norm{v}_{BV}$. Hence the final-tail reward functional has the same leading asymptotic as the profile-weighted tail integral with profile $\Phi_\psi$.
	The domination estimate from the proof of Proposition~\ref{prop:observable-domination},
	\begin{equation*}
		\abs{S_t\psi(y)}
		\le
		C\left(1+\log\frac{N+1}{N-t+1}\right),
	\end{equation*}
	is integrable against the limiting density $u^{-1-1/\gamma}\,du$ on $u>1$. Since the profile \eqref{eq:psi-tail-profile} is allowed in Definition~\ref{def:tail-profile-regular}, dominated convergence encoded by tail-profile regularity yields
	\begin{equation*}
		\Tail_{\psi,t}(v)=a_t c_{H,\psi}v(r_1+)+o(a_t)
	\end{equation*}
	for a finite constant $c_{H,\psi}\in\R$ depending on the limiting neutral reward profile, the entrance trace of $\psi$, and the regular hole. Thus Hypothesis~\ref{hyp:reward} holds with
	\begin{equation*}
		\ell_\psi(v)=c_{H,\psi}v(r_1+).
	\end{equation*}
\end{proof}

Combining the abstract reward theorems with the PM verification gives the following application-level statement.

\begin{corollary}[Bounded and convergent survivor-conditioned observables]\label{cor:general-observable}
	Let $f$, $H$, and $v\in BV(Y)$ be as in Theorem~\ref{thm:MP-renewal}, with $v\ge0$ and $\ell(\Renew v)>0$.
	\begin{enumerate}[label=(\alph*),leftmargin=2.5em,itemsep=0.4\baselineskip]
		\item If $\psi$ satisfies the hypotheses of Proposition~\ref{prop:observable-domination}, then
		 \begin{equation*}
			 \sup_{t\ge1}\abs{\E_v[S_t\psi\mid\tau_H>t]}<\infty.
		 \end{equation*}

		\item If, in addition, $f$ has asymptotically regular behavior near the neutral fixed point, $H$ is tail-profile regular, and $\psi$ satisfies \eqref{eq:psi-asymptotic}, then
		 \begin{equation*}
			 \E_v[S_t\psi\mid\tau_H>t]
		 \end{equation*}
		 converges to a finite limit.
	\end{enumerate}
\end{corollary}

\begin{proof}
	Part (a) combines Theorem~\ref{thm:MP-renewal}, Proposition~\ref{prop:observable-domination}, and Theorem~\ref{thm:observable-bounded}. Part (b) combines Theorem~\ref{thm:MP-renewal}, Proposition~\ref{prop:observable-convergence}, and Theorem~\ref{thm:observable-limit}.
\end{proof}

The Lyapunov observable is the main example. Its neutral behavior is controlled directly by the curvature assumptions on the neutral branch.

\begin{corollary}[Survivor-conditioned Lyapunov stretching]\label{cor:lyap}
	Let $f$, $H$, and $v$ be as in Corollary~\ref{cor:general-observable}. Then
	\begin{equation*}
		\sup_{t\ge1}
		\E_v\left[
			\sum_{k=0}^{t-1}\log\abs{f'(f^kx)}
			\;\middle|\;\tau_H>t
			\right]<\infty.
	\end{equation*}
	Consequently,
	\begin{equation*}
		\sup_{t\ge1}\sup_{0\le n\le t}
		\E_v\left[
			\sum_{k=0}^{n-1}\log\abs{f'(f^kx)}
			\;\middle|\;\tau_H>t
			\right]<\infty.
	\end{equation*}
	If, in addition, $f$ has asymptotically regular behavior near the neutral fixed point and $H$ is tail-profile regular, then the diagonal expectations
	\begin{equation*}
		\E_v\left[
			\sum_{k=0}^{t-1}\log\abs{f'(f^kx)}
			\;\middle|\;\tau_H>t
			\right]
	\end{equation*}
	converge to a finite limit.
\end{corollary}

\begin{proof}
	Near $0$, \textup{(\ref{MP-neutral-curvature})} implies
	\begin{equation*}
		f'(x)-1=\int_0^x f''(u)\,du\asymp x^\gamma.
	\end{equation*}
	Also
	\begin{equation*}
		\frac{d}{dx}\log\abs{f'(x)}=\frac{f''(x)}{f'(x)}=O(x^{\gamma-1})
	\end{equation*}
	on the neutral branch. Therefore $\psi=\log\abs{f'}$ satisfies the hypotheses of Proposition~\ref{prop:observable-domination}. Since $\log\abs{f'}\ge0$, the two-parameter bound follows from Corollary~\ref{cor:two-parameter}.

	If \eqref{eq:neutral-asymptotic} holds, then
	\begin{equation*}
		\log\abs{f'(x)}\sim \frac{c_f}{\gamma}x^\gamma,
	\end{equation*}
	so \eqref{eq:psi-asymptotic} holds for $\psi=\log\abs{f'}$. Proposition~\ref{prop:observable-convergence} and Theorem~\ref{thm:observable-limit} give convergence.
\end{proof}

\section{Holes containing the neutral fixed point}\label{sec:neutral-hole}

The preceding results concern holes bounded away from the neutral fixed point and satisfying the induced regularity assumptions. If the hole contains a neighborhood of the neutral fixed point, then every survivor orbit segment avoids that neighborhood, and the behavior is opposite: the orbit is forced into a uniformly expanding region.

\begin{proposition}[Linear stretching for holes containing the neutral fixed point]\label{prop:linear-hole}
	Let $f$ be a piecewise $C^1$ intermittent interval map with a neutral fixed point at $0$. Suppose that $f$ is uniformly expanding away from $0$ in the following sense: for every $a>0$ there is $\lambda_a>1$ such that
	\begin{equation}\label{eq:uniform-away-from-zero}
		\abs{f'(x)}\ge\lambda_a
	\end{equation}
	for all $x\in[a,1]$ away from partition boundaries, with one-sided derivatives at partition endpoints. If $H\supset[0,a]$ for some $a>0$, then there is $c_a>0$ such that every survivor orbit segment on which the derivatives are defined obeys
	\begin{equation*}
		S_t\log\abs{f'}(x)\ge c_at.
	\end{equation*}
	Consequently, for any absolutely continuous initial probability measure \(\nu\) with \(\nu(\{\tau_H>t\})>0\),
	\begin{equation*}
		\E_\nu[S_t\log\abs{f'}\mid\tau_H>t]\ge c_at.
	\end{equation*}
\end{proposition}

\begin{proof}
	If $\tau_H(x)>t$, then $f^kx\notin[0,a]$ for $0\le k<t$. Thus the orbit segment remains in a uniformly expanding region. At points where the derivatives along the orbit segment are defined, \eqref{eq:uniform-away-from-zero} gives
	\begin{equation*}
		S_t\log\abs{f'}(x)
		\ge t\log\lambda_a.
	\end{equation*}
	Take $c_a=\log\lambda_a$. The orbit segments which meet partition boundaries form a finite or countable union of preimages of boundary points and hence have Lebesgue measure zero. Therefore the same inequality holds after averaging with respect to any absolutely continuous initial probability measure, provided the survivor set has positive measure.
\end{proof}

\begin{remark}
	Point holes are not covered by the renewal theorem. Related problems for holes shrinking around the indifferent fixed point have a different asymptotic character; see~\cite{BonannoTikekar2026}. Indeed, if $H=\{z\}$ is a singleton, then for the maps considered here each finite preimage set $f^{-k}(H)$ is finite, hence Lebesgue-null. Therefore every absolutely continuous initial measure assigns zero mass to
	\[
		\bigcup_{k=0}^{t-1}f^{-k}(H),
	\]
	so no positive mass escapes through a singleton hole at any finite time. In this sense a point hole produces no nontrivial escape problem for the absolutely continuous initial distributions considered here. In particular, the singleton hole $\{0\}$ is not covered by Proposition~\ref{prop:linear-hole}, since that proposition requires the hole to contain a neighborhood of the neutral fixed point.
\end{remark}

\section{Discussion}

The main theorem provides a deterministic renewal theorem behind the numerical and stochastic picture of~\cite{BrevittKlages2025}. Under the survivor-conditioned distribution, the number $N_t$ of completed survivor returns converges in law and hence remains tight. The leading contribution to survival is therefore described by a bounded number of completed survivor returns followed by a final return interval whose length is comparable to the observation time. Once this decomposition is proved, several consequences follow: polynomial survival asymptotics, the limiting distribution of the number of completed survivor returns, a geometric limiting law in the conditionally invariant case, boundedness of survivor-conditioned Birkhoff sums for observables satisfying the reward domination hypotheses, convergence for observables with a regular final-tail asymptotic, bounded Lyapunov stretching in the generalized PM application, convergence of Lyapunov stretching when the behavior near the neutral fixed point and the final tails are asymptotically regular, and, under an entropy-domination hypothesis, zero entropy rate for return-length names.

This separates the survivor-conditioned renewal structure studied here from earlier work on open intermittent maps. Demers--Fernandez and Demers--Todd describe escape, density collapse, conditionally invariant structures, and thermodynamic regimes; for broader background on escape rates, conditionally invariant measures, and hitting and escaping statistics see~\cite{DemersYoung2006,DemersWrightYoung2010,BruinDemersTodd2018}. Our main theorem is a pathwise refinement: it gives the renewal decomposition of survivor-conditioned orbit segments and the associated reward estimates. The price for the sharper limiting law is the regular induced-hole assumption. Rough positive-measure visibility is enough for boundedness estimates, but sharp fixed-$r$ asymptotics require enough regularity to control the killed induced transfer operators and the final-tail functional. Likewise, neutral-branch size estimates for observables give bounded survivor-conditioned rewards, whereas convergence requires a limiting asymptotic profile.

\end{document}